\documentclass[reqno]{amsart}
\usepackage[sc]{mathpazo} 

\usepackage{amsmath}
\usepackage{amsfonts}
\usepackage{amsthm}
\usepackage{amssymb}
\usepackage{graphicx,psfrag,epsfig}
\usepackage{color}
\usepackage{mathrsfs}
\usepackage{pdfsync}
\usepackage{cite}
\usepackage{times}

\newtheorem{thm}{Theorem}[section]
\newtheorem{lem}[thm]{Lemma}
\newtheorem{cor}[thm]{Corollary}
\newtheorem{prop}[thm]{Proposition}

\newtheorem{rem}[thm]{Remark}

\DeclareMathAlphabet{\mathpzc}{OT1}{pzc}{m}{it}

\numberwithin{equation}{section}

\newcommand{\R}{\mathbb{R}}

\newcommand{\ve}{\varepsilon}
\newcommand{\e}{\varepsilon}
\newcommand{\rd}{\mathrm{d}}

\newcommand{\dhr}{\mathrel{\lhook\joinrel\relbar\kern-.8ex\joinrel\lhook\joinrel\rightarrow}} 

\title[Fourth-order MEMS in three dimensions]
{On a three-dimensional free boundary problem modeling electrostatic MEMS}
\author{Philippe Lauren\c{c}ot}
\thanks{ Partially supported by the French-German PROCOPE project 30718ZG.}
\address{Institut de Math\'ematiques de Toulouse, UMR~5219, Universit\'e de Toulouse, CNRS \\ F--31062 Toulouse Cedex 9, France}
\email{laurenco@math.univ-toulouse.fr}

\author{Christoph Walker}
\address{Leibniz Universit\"at Hannover\\ Institut f\" ur Angewandte Mathematik \\ Welfengarten 1 \\ D--30167 Hannover\\ Germany}
\email{walker@ifam.uni-hannover.de}

\date{\today}
\begin{document}

%%%%%%%%%%%%%%%%%%%%%%%
\begin{abstract}
We consider the dynamics of an electrostatically actuated thin elastic plate  being clamped at its boundary above a rigid plate. The model includes the harmonic electrostatic potential in the three-dimensional time-varying region between the plates along with a fourth-order semilinear parabolic equation for the elastic plate deflection which is coupled to the square of the gradient trace of the electrostatic potential on this plate. The strength of the coupling is tuned by a parameter $\lambda$ proportional to the square of the applied voltage. We prove that this free boundary problem is locally well-posed in time and that for small values of $\lambda$ solutions exist globally in time. We also derive the existence of a branch of asymptotically stable stationary solutions for small values of $\lambda$ and non-existence of stationary solutions for large values thereof, the latter being restricted to a disc-shaped plate.
\end{abstract}
%%%%%%%%%%%%%%%%%%%%%%%

%%%%%%%%%%%%%%%%%%%%%%%
%%%%%%%%%%%%%%%%%%%%%%%

\maketitle

%%%%%%%%%%%%%%%%%%%%%%%
%%%%%%%%%%%%%%%%%%%%%%%
\section{Introduction and Main Results}
%%%%%%%%%%%%%%%%%%%%%%%
%%%%%%%%%%%%%%%%%%%%%%%

We focus on an idealized model for an electrostatically actuated microelectromechanical system (MEMS). The device is built of a thin conducting elastic plate being clamped at its boundary above a rigid conducting plate. A Coulomb force is induced across the device by holding the ground and the elastic plates at different electric potentials which results in a deflection of the elastic plate and thus in a change in geometry of the device, see Figure~\ref{MEMSfig}. An ubiquitous feature of such MEMS devices is the occurrence of the so-called ``pull-in'' instability which manifests above a critical  threshold for the voltage difference in a touchdown of the elastic plate on the rigid ground plate. Estimating this threshold value is of utmost interest in applications as it determines the stable operating regime for such a MEMS device. To set up a mathematical model we assume that the dynamics of the device can be fully described by the deflection of the elastic plate from its rest position (when no voltage difference exists) and the electrostatic potential in the varying region between the two plates. We further assume that the elastic plate in its rest position and the fixed ground plate can be described by a region $D$ in $\R^2$. After a suitable scaling the rigid ground plate is located at $z=-1$ and the rest position of the elastic plate is at $z=0$. If $u=u(t,x,y)$ for $t\ge 0$ and  $(x,y)\in D$ describes the vertical displacement of  the elastic plate from its rest position, then $u$ evolves in the damping dominated regime according to
\begin{equation}\label{u}
\partial_t u+\beta\Delta^2 u-\big(\tau+a\|\nabla u\|_2^2\big)\Delta u = - \lambda\,  |\nabla_\varepsilon\psi(t,x,y,u(t,x,y))|^2 \ ,\quad (x,y)\in D\ ,\quad t>0\,,
\end{equation}
with clamped boundary conditions
\begin{equation}\label{bcu}
u=\partial_\nu u=0\,  ,\quad (x,y)\in \partial D\ , \quad t>0\,,
\end{equation}
and initial condition
\begin{equation}\label{ic}
u(0,x,y)=u^0(x,y)\ ,\quad (x,y)\in D\ .
\end{equation}
Here we put
$$
\nabla_\varepsilon\psi:=\left(\varepsilon\partial_x\psi, \varepsilon\partial_y\psi,\partial_z\psi\right)\,,
$$
where $\ve>0$ is the aspect ratio of the device, i.e. the ratio between vertical and horizontal dimensions,  $\lambda>0$ is proportional to the square of the applied voltage difference, and $\psi=\psi(t,x,y,z)$ denotes the dimensionless electrostatic potential. The latter satisfies a rescaled Laplace equation
\begin{equation}\label{psi}
\varepsilon^2\partial_x^2\psi + \varepsilon^2\partial_y^2\psi +\partial_z^2\psi =0\ ,\quad (x,y,z)\in \Omega(u(t))\ ,\quad t>0\ ,
\end{equation}
in the cylinder
$$
\Omega(u(t)) := \left\{ (x,y,z)\in D\times (-1,\infty)\ :\ -1 < z < u(t,x,y) \right\}\,
$$
between the rigid ground plate at $z=-1$ and the deflected elastic plate. The boundary conditions for $\psi$ are 
\begin{equation}\label{bcpsi}
\psi(t,x,y,z)=\frac{1+z}{1+u(t,x,y)}\ ,\quad (x,y,z)\in  \partial\Omega(u(t))\,, \quad t>0 \  .
\end{equation}
In equation \eqref{u}, the fourth-order term $\beta \Delta^2 u$ with $\beta>0$ reflects plate bending while the linear second-order term $\tau \Delta u$ with $\tau\ge 0$ and the non-local second-order term $a\|\nabla u\|_{2}^2 \Delta u$ with $a\ge 0$ and
$$
\|\nabla u\|_{2}^2:=\int_D \vert\nabla u\vert^2\,\rd (x,y)
$$
account for external stretching and self-stretching forces generated by large oscillations, respectively. The right-hand side of \eqref{u} is due to  the electrostatic forces exerted on the elastic plate  and is tuned by the strength of the applied voltage difference which is accounted for by the parameter $\lambda$. The boundary conditions \eqref{bcu} mean that the elastic plate is clamped. According to \eqref{psi}-\eqref{bcpsi}, the electrostatic potential is harmonic in the region $\Omega(u)$ enclosed by the two plates with value~$1$ on the elastic plate and value~$0$ on the ground plate. We refer the reader e.g. to  \cite{EGG10,FargasEtAl,Bending2,PB03} and the references therein for more details on the derivation of the model.

Equations \eqref{u}-\eqref{bcpsi} feature a singularity which reflects the pull-in instability occurring when the elastic plate touches down on the ground plate. Indeed, when $u$ reaches the value $-1$ somewhere, the region $\Omega(u)$ gets disconnected. Moreover, the imposed boundary conditions \eqref{bcpsi} imply that the vertical derivative $\partial_z \psi(x,y,u(x,y))$ blows up at the touchdown point $(x,y)$ and, in turn, the right-hand side of \eqref{u} becomes singular. Questions regarding (non-)existence of stationary solutions and of global solutions to the evolution problem as well as the qualitative behavior of the latter are strongly related. Due to the intricate coupling of the possibly singular equation  \eqref{u} and the free boundary problem \eqref{psi}-\eqref{bcpsi} in non-smooth domains, answers are, however, not easy to obtain.

%%%%%%%%%%%%%%%%%%%%%%%
\begin{figure}
%\centering\includegraphics[scale=.7]{MEMS2d_A.pdf}
\centering\includegraphics[scale=.7]{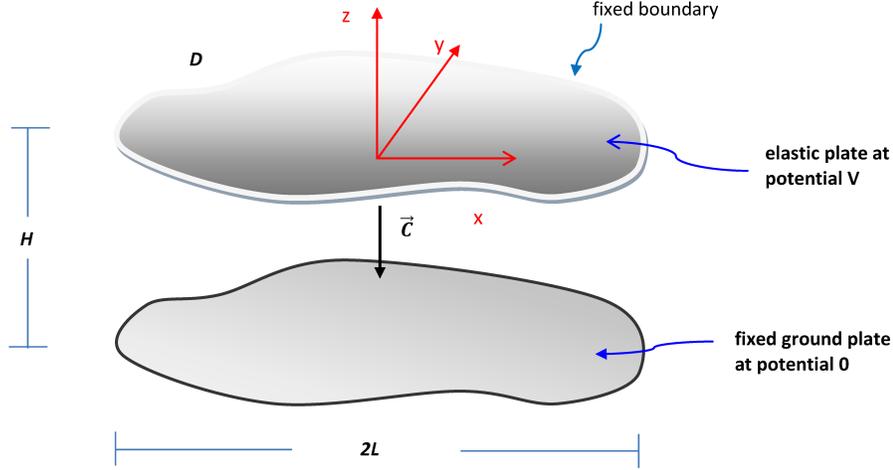}
\caption{\small Schematic diagram of an idealized electrostatic MEMS device.}\label{MEMSfig}
\end{figure}
%%%%%%%%%%%%%%%%%%%%%%%

It is thus not surprising that most mathematical research to date has been dedicated to various variants of the so-called {\it small gap model}, a relevant approximation of \eqref{u}-\eqref{bcpsi} obtained by formally setting $\ve=0$ therein. This approximation allows one to compute the electrostatic potential 
$$
\psi(t,x,z)=\frac{1+z}{1+u(t,x,y)}
$$ 
explicitly in dependence of $u$, the latter to be determined via 
\begin{equation*}
\partial_t u+\beta\Delta^2 u-\big(\tau+a\|\nabla u\|_2^2\big)\Delta u = - \lambda\, \frac{1}{(1+ u)^2}\ ,\quad (x,y)\in D\ , \quad t>0\,,
\end{equation*}
subject to \eqref{bcu} and \eqref{ic}. We note that this small gap approximation is a  singular equation but no longer a free boundary problem. We refer to \cite{EGG10, GLY14, LW14c, LL12} and the references therein for  more information on this case. 

The free boundary problem \eqref{u}-\eqref{bcpsi} has been investigated in a series of papers by the authors \cite{ELW1, LW13, Bending1, Bending2, LWxx}, though in a simpler geometry assuming $D$ to be a rectangle and presupposing zero variation in the $y$-direction, see also \cite{Li15} for the case of a non-constant permittivity profile. In this geometry the deflection $u=u(t,x)$ is independent of $y$ and $\psi$ is harmonic in a two-dimensional domain. In the present paper we remove this assumption and tackle for the first time the evolution problem with a three-dimensional domain $\Omega(u(t))$, assuming only a convexity property on $D$. More precisely, we assume in the following that
\begin{equation}\label{D}
D \text{ is a bounded and convex domain in } \R^2 \text{ with a $C^2$-smooth boundary.}
\end{equation}
A typical example for $D$ is a disc. We shall see later on that condition~\eqref{D} is used to obtain sufficiently smooth solutions to the elliptic problem \eqref{psi}-\eqref{bcpsi} in order for the trace of $\nabla_\varepsilon\psi$ to be well-defined on $\partial\Omega(u(t))$, a fact which is not clear at first glance as $\Omega(u(t))$ is only a Lipschitz domain. Still it turns out that the regularity of the square of the gradient trace of this solution occurring on the right-hand side of \eqref{u} is weaker than in the case of a two-dimensional domain $\Omega(u(t))$ studied in \cite{ELW1, LW13, Bending1, Bending2, LWxx} restricting us to study the fourth-order case $\beta>0$ only, see Remark~\ref{R1} for more information.

\medskip

The following result shows that \eqref{u}-\eqref{bcpsi} is locally well-posed in general and globally well-posed for small $a$ or small initial values provided that $\lambda$ is small as well.

%%%%%%%%%%%%%%%%%%%%%%%
\begin{thm}[{\bf Local and Global Well-Posedness}]\label{A}
Suppose \eqref{D}. Let $\ve>0$, $4\xi\in (7/3,4)$, and consider an initial value $u^0\in  W_{2}^{4\xi}(D)$ such that $u^0(x,y)>-1$ for $(x,y)\in D$ and $u^0=\partial_\nu u^0=0$ on $\partial D$. Then, the following are true:

\begin{itemize}

\item[(i)] For each voltage value $\lambda>0$, there is a unique solution $(u,\psi)$ to \eqref{u}-\eqref{bcpsi} on the maximal interval of existence $[0,T_m)$ in the sense that
\begin{equation}\label{reg}
u\in C\big([0,T_m), W_{2}^{4\xi}(D)\big) \cap C\big((0,T_m), W_{2}^4(D)\big) \cap C^1\big((0,T_m),L_2(D)\big) 
\end{equation}
satisfies \eqref{u}-\eqref{ic} together with
$$
u(t,x,y)>-1\ ,\quad (t,x,y)\in [0,T_m)\times D\ , 
$$ 
and $\psi(t)\in W_2^2\big(\Omega(u(t))\big)$ solves \eqref{psi}-\eqref{bcpsi} in $\Omega(u(t))$ for each $t\in [0,T_m)$. 

\item[(ii)] If, for each $T>0$, there is $\kappa(T)\in (0,1)$ such that 
$$
\|u(t)\|_{W_2^{4\xi}(D)}\le\kappa(T)^{-1}\ ,\quad  u(t)\ge -1+\kappa(T)\ \ \text{in}\ \ D
$$
for $t\in [0,T_m)\cap [0,T]$, then the solution exists globally, that is, $T_m=\infty$.

\item[(iii)] 
Given $\kappa\in (0,1/2)$, there exists $m_*:=m_*(\kappa,\varepsilon)>0$ such that $T_m=\infty$ and 
$$
\|u(t)\|_{W_2^{4\xi}(D)}\le\kappa^{-1}\ ,\quad  u(t)\ge -1+\kappa\ \text{in}\ D
$$ 
for $t\ge 0$, provided that $\lambda+ a\|\nabla u^0\|_{2}^2\le m_*$ and 
$$
\|u^0\|_{W_2^{4\xi}(D)}\le (2\kappa)^{-1}\ ,\quad  u^0\ge -1+2\kappa\ \text{in}\ D \ .
$$ 

\item[(iv)]  If $D$ is a disc in $\R^2$ and $u^0=u^0(x,y)$ is radially symmetric with respect to $(x,y)\in D$, then, {for all $t\in [0,T_m)$}, $u=u(t,x,y)$ and $\psi=\psi(t,x,y,z)$ are radially symmetric with respect to~$(x,y)\in D$.

\end{itemize}
\end{thm}
%%%%%%%%%%%%%%%%%%%%%%%

The global existence criterion stated in part (ii) of Theorem~\ref{A} involving a blow-up of some Sobolev-norm or occurrence of a touchdown is not yet optimal as the possible norm blow-up has rather mathematical than physical reasons. In the case of a two-dimensional domain $\Omega(u)$ this condition is superfluous as shown in \cite{Bending1}. Note that part (iii) of Theorem~\ref{A} provides uniform estimates on $u$ and ensures, in particular, that $u$ never touches down on $−1$, not even in infinite time.

Regarding existence of stationary solutions to \eqref{u}-\eqref{bcpsi} for an arbitrary convex domain $D$ we have:

%%%%%%%%%%%%%%%%%%%%%%%
\begin{thm}[{\bf Stationary Solutions}]\label{Stat}
Suppose \eqref{D}. Then, given $\ve>0$ and $\kappa\in (0,1)$, there are \mbox{$\delta:=\delta(\kappa,\ve)>~0$} and an analytic function $[\lambda\mapsto U_\lambda]:[0,\delta)\rightarrow W_2^4(D)$ such that $(U_\lambda,\Psi_\lambda)$ is for each $\lambda\in (0,\delta)$ an asymptotically stable stationary solution to \eqref{u}-\eqref{bcpsi} with $U_\lambda > -1+\kappa$ in $D$ and $\Psi_\lambda\in W_2^2(\Omega(U_\lambda))$. 
\end{thm}
%%%%%%%%%%%%%%%%%%%%%%%

We shall point out that while Theorem~\ref{Stat} ensures the existence of at least one stationary solution for a fixed, sufficiently small voltage value $\lambda$, a recent result \cite{LWxx} yields a second one for (some of) these values in the two-dimensional case.

When $D$ is a disc in $\R^2$, additional information on stationary solutions can be retrieved, in particular a non-existence result when $a=0$. Roughly speaking, this additional information is provided by the fact that the operator $\beta\Delta^2-\tau\Delta$ satisfies the maximum principle when restricted to radially symmetric functions (see Section~\ref{Sec5} for more details).

%%%%%%%%%%%%%%%%%%%%%%%
\begin{thm}\label{StatDisc}
Assume that $D$ is a disc in $\R^2$ and let $\ve>0$.
\begin{itemize}
\item[(i)] For any $\kappa\in (0,1)$ and $\lambda\in (0,\delta(\kappa,\ve))$, the stationary solution $(U_\lambda,\Psi_\lambda)$ to \eqref{u}-\eqref{bcpsi} constructed in Theorem~\ref{Stat} enjoys the following properties: the function $U_\lambda$ is radially symmetric and $U_\lambda<0$ in $D$ while $\Psi_\lambda$ is radially symmetric in the first two variables $(x,y)$.

\item[(ii)] Assume that $a=0$. There are $\varepsilon_*>0$ and a function $\Lambda: (0,\varepsilon_*) \to (0,\infty)$ such that there is no radially symmetric stationary solution $(u,\psi)$ to \eqref{u}-\eqref{bcpsi} for $\lambda>\Lambda(\varepsilon)$ and $\varepsilon\in (0,\varepsilon_*)$.
\end{itemize}
\end{thm}
%%%%%%%%%%%%%%%%%%%%%%%

The outline of this paper is as follows. In Section~\ref{Sec2} we first investigate the elliptic problem \eqref{psi}-\eqref{bcpsi} for the electrostatic potential in dependence on a given deflection of the elastic plate. The main result of this section is Proposition~\ref{L1} which implies that \eqref{u}-\eqref{bcpsi} can be rewritten as a semilinear equation for the deflection $u$ only. The proof is rather involved and divided into several steps. Section~\ref{Sec3} is then devoted to the proof of Theorem~\ref{A} while the proof of Theorem~\ref{Stat} is given in Section~\ref{Sec4}. Finally, in Section~\ref{Sec5} we indicate how the proof of Theorem~\ref{StatDisc} can be carried out based on the corresponding proof for the two-dimensional case.

%%%%%%%%%%%%%%%%%%%%%%%
%%%%%%%%%%%%%%%%%%%%%%%
\section{The Electrostatic Potential Equation}\label{Sec2}
%%%%%%%%%%%%%%%%%%%%%%%
%%%%%%%%%%%%%%%%%%%%%%%

We first focus on the free boundary problem \eqref{psi}-\eqref{bcpsi} which we transform to the cylinder 
$$
\Omega:=D\times (0,1)\,.
$$
 More precisely, let $q\ge 2$ be fixed and consider an arbitrary function $v\in W_{q,\mathcal{B}}^2(D)$ taking values in $(-1,\infty)$, where
$$
W_{p,\mathcal{B}}^{\alpha}(D):=\left\{\begin{array}{ll}
\{w\in  W_{p}^{\alpha}(D)\ :\  w=\partial_\nu w=0 \text{  on } \partial D\}\ , & \alpha > 1+1/p\,,\\
& \\
\{w\in  W_{p}^{\alpha}(D)\ :\ w=0 \text{  on } \partial D\}\ , &\alpha \in (1/p,1+1/p)\,,\\
&\\
W_{p}^{\alpha}(D)\ , &0\le \alpha < 1/p\ ,\end{array}\right.
$$
for $p\ge 2$. We define 
$$
\Omega(v) := \left\{ (x,y,z)\in D\times (-1,\infty)\ :\ -1 < z < v(x,y) \right\}
$$
and consider the rescaled Laplace equation
\begin{eqnarray}
\varepsilon^2\partial_x^2\psi_v + \varepsilon^2\partial_y^2\psi_v +\partial_z^2\psi_v & = & 0\ ,\quad (x,y,z)\in \Omega(v)\ , \label{psiv} \\
\psi_v(x,y,z) & = & \frac{1+z}{1+v(x,y)}\ , \quad (x,y,z)\in \partial\Omega(v)\ . \label{bcpsiv}
\end{eqnarray}
Introducing the diffeomorphism $\mathcal{T}_v:\overline{\Omega(v)}\rightarrow \overline{\Omega}$ given by
\begin{equation}\label{Tu}
\mathcal{T}_v(x,y,z):=\left(x,y,\frac{1+z}{1+v(x,y)}\right)\ ,\quad (x,y,z)\in \overline{\Omega(v)}\ ,
\end{equation}
we note that its inverse is
\begin{equation}\label{Tuu}
\mathcal{T}_v^{-1}(x,y,\eta)=\big(x,y,(1+v(x,y))\eta-1\big)\ ,\quad (x,y,\eta)\in \overline{\Omega}\ ,
\end{equation}
and that the rescaled Laplace operator in \eqref{psiv} is transformed to the $v$-dependent differential operator 
\begin{equation*}
\begin{split}
\mathcal{L}_v w\, :=\, & \e^2\ \partial_x^2 w + \e^2\ \partial_y^2 w - 2\e^2\ \eta\ \frac{\partial_x v(x,y)}{1+v(x,y)}\ \partial_x\partial_\eta w - 2\e^2\ \eta\ \frac{\partial_y v(x,y)}{1+v(x,y)}\ \partial_y\partial_\eta w\\&
+ \frac{1+\e^2\eta^2 |\nabla v(x,y)|^2}{(1+v(x,y))^2}\ \partial_\eta^2 w + \e^2\ \eta\ \left[ 2\ \frac{|\nabla v(x,y)|^2}{(1+v(x,y))^2} - \frac{\Delta v(x,y)}{1+v(x,y)} \right]\ \partial_\eta w\ .
\end{split}
\end{equation*}
The boundary value problem  \eqref{psiv}-\eqref{bcpsiv} is then equivalent to
\begin{eqnarray}
\big(\mathcal{L}_{v} \phi_v\big) (x,y,\eta)\!\!\!&=0\ ,&(x,y,\eta)\in\Omega\ , \label{23}\\
\phi_v(x,y,\eta)\!\!\!&=\eta\ , &(x,y,\eta)\in \partial\Omega\ , \label{24}
\end{eqnarray}
via the transformation $\phi_v=\psi_v\circ \mathcal{T}_{v}^{-1}$. Observe that \eqref{23}-\eqref{24} is an elliptic equation with non-constant coefficients but in the fixed domain $\Omega$.

\medskip

We now aim at studying precisely the well-posedness of \eqref{23}-\eqref{24} as well as the regularity of its solutions. To this end, we introduce for $\kappa\in (0,1)$ and $p\ge 2$ the set
\begin{align*}
S_p(\kappa)  :=\left\{ w\in W_{p,\mathcal{B}}^2(D)\,:\, \|w\|_{W_{p,\mathcal{B}}^2(D)}< 1/\kappa \;\;\text{ and }\;\; -1+\kappa< w(x,y) \text{ for } (x,y)\in D \right\}\ , 
\end{align*}
with
\begin{align*}
\overline{S}_p(\kappa) & :=\left\{ w\in W_{p,\mathcal{B}}^2(D)\,:\, \|w\|_{W_{p,\mathcal{B}}^2(D)} \le 1/\kappa \;\;\text{ and }\;\; -1+\kappa \le w(x,y) \text{ for } (x,y)\in D \right\}\ 
\end{align*}
being its closure in $W_{p,\mathcal{B}}^2(D)$. The key result of this section reads:

%%%%%%%%%%%%%%%%%%%%%%%
\begin{prop}\label{L1}
Suppose \eqref{D}. Let $\ve>0$, $\kappa\in (0,1)$, and $q\ge 3$. For each $v\in S_q(\kappa)$ there is a unique solution $\phi_v\in W_2^2(\Omega)$ to \eqref{23}-\eqref{24}. Furthermore there is $c_1(\kappa,\ve)>0$ such that 
\begin{align}
\|\phi_{v_1}-\phi_{v_2}\|_{W_{2}^2(\Omega)} 
\le \ c_1(\kappa,\varepsilon)\ \|v_1-v_2\|_{W_q^2(D)}\ , \quad v_1, v_2 \in S_q(\kappa)\ ,\label{RR}
\end{align}
and the mapping
$$
g_\ve: S_q(\kappa)\longrightarrow L_2(D)\ ,\quad v\longmapsto \frac{1+\e^2 |\nabla v|^2}{(1+v)^2}\  \vert\partial_\eta\phi_v(\cdot,\cdot,1)\vert^2
$$
is analytic, globally Lipschitz continuous, and bounded. 
\end{prop}
%%%%%%%%%%%%%%%%%%%%%%%

Several steps are needed to prove Proposition~\ref{L1}. We begin with the analysis of the Dirichlet problem associated to $\mathcal{L}_v$. 

%%%%%%%%%%%%%%%%%%%%%%%
\begin{lem}\label{L2}
Suppose \eqref{D}. Let $\varepsilon>0$, $\kappa\in (0,1)$, and $q>2$. For each $v\in \overline{S}_q(\kappa)$ and $F\in L_2(\Omega)$, there exists a unique solution 
$$
\Phi\in W_{2,\mathcal{B}}^1(\Omega) := \{ w \in W_2^1(\Omega)\ :\ w = 0 \text{ on } \partial\Omega\}
$$ 
to the boundary value problem 
\begin{equation}
-\mathcal{L}_v\Phi =F\ \text{ in }\ \Omega\ , \qquad \Phi  =0\  \text{ on }\ \partial\Omega\ . \label{231}
\end{equation}
\end{lem}
%%%%%%%%%%%%%%%%%%%%%%%

\begin{proof}[{\bf Proof}]
Since $q>2$, the definition of $\overline{S}_q(\kappa)$ and Sobolev's embedding theorem guarantee the existence of some constant $c_0>0$ depending only on $q$ and $D$ such that, for $v\in \overline{S}_q(\kappa)$,
\begin{equation}\label{c2}
1+v(x,y)\ge \kappa\ ,\quad (x,y)\in D\ ,\quad\text{ and }\;\; \|v\|_{C^1(\bar D)}\le \frac{c_0}{\kappa}\ .
\end{equation}
We now claim that, due to \eqref{c2}, the operator $-\mathcal{L}_v$ is elliptic with ellipticity constant $\nu(\kappa,\varepsilon)>0$ being independent of $v\in  \overline{S}_q(\kappa)$. Indeed, for $(x,y,\eta)\in\Omega$, set
\begin{align*}
&a_{11}(v) :=  \varepsilon^2\,, \qquad & \qquad  a_{13}(v) = a_{31}(v) := - \varepsilon^2\ \eta\ \left( \frac{\partial_x v}{1+v} \right)(x,y)\,,  \\
&a_{22}(v):= \varepsilon^2\,, \qquad & \qquad a_{23}(v) = a_{32}(v) := - \varepsilon^2\ \eta\ \left( \frac{\partial_y v}{1+v} \right)(x,y)\,,  \\
& a_{33}(v)  := \left( \frac{1+\varepsilon^2\ \eta^2 |\nabla v|^2}{(1+v)^2} \right)(x,y)\,,  & b_1(v) := \varepsilon^2\ \left( \frac{\partial_x v}{1+v} \right)(x,y)\,, \\
& b_2(v) := \varepsilon^2\ \left( \frac{\partial_y v}{1+v} \right)(x,y)\,, 
& b_3(v) := - \varepsilon^2\ \eta \left( \frac{|\nabla v|^2}{(1+v)^2} \right)(x,y)\,, 
\end{align*}
which allows us to write $-\mathcal{L}_v$ in divergence form:
\begin{align*}
-\mathcal{L}_v w =& -\partial_x \left( a_{11}(v)\ \partial_x w + a_{13}(v)\ \partial_\eta w \right)
-\partial_y \left( a_{22}(v)\ \partial_y w + a_{23}(v)\ \partial_\eta w \right)\\
& -\partial_\eta \left( a_{31}(v)\ \partial_x w + a_{32}(v)\ \partial_y w+ a_{33}(v)\ \partial_\eta w \right)\\
& - b_1(v)\partial_x w -b_2(v)\partial_y w -b_3(v)\partial_\eta w\ .
\end{align*}
Denoting the principal of $-\mathcal{L}_v$ by $-\mathcal{L}_v^0$, that is,
\begin{align*}
-\mathcal{L}_v^0 w :=& -\partial_x \left( a_{11}(v)\ \partial_x w + a_{13}(v)\ \partial_\eta w \right)
-\partial_y \left( a_{22}(v)\ \partial_y w + a_{23}(v)\ \partial_\eta w \right)\\
& -\partial_\eta \left( a_{31}(v)\ \partial_x w + a_{32}(v)\ \partial_y w+ a_{33}(v)\ \partial_\eta w \right)\ ,
\end{align*}
and introducing the associated matrix
$$
\mathcal{P}:=\left(\begin{matrix} 
 \e^2 & 0 & \displaystyle{-\frac{\e^2\partial_x v(x,y)}{1+v(x,y)}\eta} \\ 
 0 & \e^2 & \displaystyle{-\frac{\e^2\partial_y v(x,y)}{1+v(x,y)}\eta}\\
\displaystyle{-\frac{\e^2\partial_x v(x,y)}{1+v(x,y)}\eta} & 
\displaystyle{-\frac{\e^2\partial_y v(x,y)}{1+v(x,y)}\eta}
& \displaystyle{\frac{1+\e^2\eta^2 |\nabla v(x,y)|^2}{(1+v(x,y)^2}}
\end{matrix}\right)
$$
we observe that, for $(x,y,\eta)\in\Omega$, the eigenvalues of $\mathcal{P}$ are $\e^2$ and
$$
\mu_\pm =\frac{1}{2}\big( t\pm\sqrt{t^2-4d}\big)
$$
with  $t$ and $d$ given by
$$
t:= \e^2 + \frac{1+\e^2\eta^2 |\nabla v(x,y)|^2}{(1+v(x,y))^2}\ ,\qquad d:= \frac{\e^2}{(1+v(x,y))^2}\ .
$$
By \eqref{c2}, 
$$
\mu_+ > \mu_- \ge \frac{d}{t} \ge \frac{\kappa \ve^2}{(1+\ve^2) \kappa^2 + 2 \ve^2 c_0^2}>0\ ,
$$
which implies that $-\mathcal{L}_v$ is elliptic with a positive ellipticity constant depending on $\kappa$ and $\e$ only. Furthermore, we infer from \eqref{c2} and the definition of $\overline{S}_q(\kappa)$ that 
\begin{equation}
\sum_{i,j=1}^3 \|a_{ij}(v)\|_{L_\infty(\Omega)} + \sum_{i=1}^3 \|b_i(v)\|_{L_\infty(\Omega)} \le c_2(\kappa,\varepsilon) \label{y4b}
\end{equation} 
for all $v\in \overline{S}_q(\kappa)$. It then follows from \cite[Theorem~8.3]{GilbargTrudinger} that, given $v\in \overline{S}_q(\kappa)$ and $F\in L_2(\Omega)$, the boundary value problem \eqref{231} has a unique weak solution $\Phi\in W_{2,\mathcal{B}}^1(\Omega)$. 
\end{proof}

Next, for smoother functions $v$, we make use of the convexity of $\Omega$ to gain more regularity on the solution to \eqref{231}.

%%%%%%%%%%%%%%%%%%%%%%%
\begin{lem}\label{L2a}
Suppose \eqref{D}. Let $\varepsilon>0$ and $\kappa\in (0,1)$. For each $v\in \overline{S}_\infty(\kappa)$ and $F\in L_2(\Omega)$, the weak solution $\Phi\in W_{2,\mathcal{B}}^1(\Omega)$ to \eqref{231} belongs to $W_{2,\mathcal{B}}^2(\Omega) := W_2^2(\Omega) \cap W_{2,\mathcal{B}}^1(\Omega)$.
\end{lem}
%%%%%%%%%%%%%%%%%%%%%%%

\begin{proof}[{\bf Proof}]
Consider $F\in L_2(\Omega)$ and denote the corresponding weak solution to \eqref{231} by $\Phi \in W_{2,\mathcal{B}}^1(\Omega)$. The regularity of $\Phi$ and $v$  ensure that 
$$
G:= F + b_1(v)\ \partial_x\Phi +b_2(v)\ \partial_y\Phi + b_3(v)\ \partial_\eta\Phi \in L_2(\Omega)\ .
$$ 
Since $\Omega$ is convex and $a_{ij}(v) \in W_\infty^1(\Omega)$ for $1\le i,j\le 3$, we are in a position to apply \cite[Theorem~3.2.1.2]{Grisvard} to conclude that there is a unique solution $\widehat{\Phi}\in W_{2,\mathcal{B}}^2(\Omega)$ to the boundary value problem
\begin{equation}\label{AA}
-\mathcal{L}_v^0 \widehat{\Phi}  =G\ \text{ in }\ \Omega\ ,\qquad
\widehat{\Phi} =0\  \text{ on }\ \partial\Omega\ .
\end{equation}
where $\mathcal{L}_v^0$ is the principal part of the operator $\mathcal{L}_v$. It also follows from \cite[Theorem~8.3]{GilbargTrudinger} that \eqref{AA} has a unique weak solution in $W_{2,\mathcal{B}}^1(\Omega)$. Due to \eqref{231} and the definition of $G$,
the functions $\Phi$ and $\widehat{\Phi}$ are both weak solutions in $W_{2,\mathcal{B}}^1(\Omega)$ to \eqref{AA} and thus $\Phi=\widehat{\Phi}\in W_{2,\mathcal{B}}^2(\Omega)$. 
\end{proof}

The next step is to adapt the analysis performed in the two-dimensional case in \cite[Section~4]{LWxx} to derive an estimate on the $W_2^1(\Omega)$-norm of $\partial_\eta \Phi$ which is suitably uniform with respect to $v$. We begin with the following trace estimate.

%%%%%%%%%%%%%%%%%%%%%%%
\begin{lem}\label{le.lazy}
Suppose \eqref{D}. Let $p\in [2,4]$. There exists $c_3(p)>0$ such that
\begin{equation}
\|w(\cdot,\cdot,1)\|_{L_p(D)}^p \le c_3(p) \|w\|_{W_2^1(\Omega)}^{(3p-4)/2}  \| w\|_{L_2(\Omega)}^{(4-p)/2}\ , \qquad w\in W_2^1(\Omega)\ . \label{lazy00}
\end{equation}
\end{lem}
%%%%%%%%%%%%%%%%%%%%%%%

\begin{proof}[{\bf Proof}] Let $w\in W_2^1(\Omega)$. For $\eta\in (0,1)$ and $(x,y)\in D$, one has
$$
|w(x,y,1)|^p = |w(x,y,\eta)|^p + p \int_\eta^1 |w(x,y,z)|^{p-2} w(x,y,z)\ \partial_\eta w(x,y,z)\,\rd z\ .
$$
Integrating the above identity with respect to $(x,y)\in D$ gives
\begin{align*}
\|w(\cdot,\cdot,1)\|_{L_p(D)}^p & \le \int_D |w(x,y,\eta)|^p \,\rd(x,y) \\
& \quad + p \int_\Omega |w(x,y,z)|^{p-1} |\partial_\eta w(x,y,z)|\,\rd (x,y,z)\ .
\end{align*}
We next integrate with respect to $\eta\in (0,1)$ and use H\"older's inequality to obtain
\begin{align*}
\|w(\cdot,\cdot,1)\|_{L_p(D)}^p & \le \|w\|_{L_p(\Omega)}^p + p \|w\|_{L_{2(p-1)}(\Omega)}^{p-1} \|\partial_\eta w\|_{L_2(\Omega)} \\
& \le \|w\|_2 \|w\|_{L_{2(p-1)}(\Omega)}^{p-1} + p \|w\|_{L_{2(p-1)}(\Omega)}^{p-1} \|\partial_\eta w\|_{L_2(\Omega)} \\
& \le p \|w\|_{L_{2(p-1)}(\Omega)}^{p-1} \| w\|_{W_2^1(\Omega)} \ .
\end{align*}
We finally use the Gagliardo-Nirenberg inequality
$$
\|w\|_{L_{2(p-1)}(\Omega)}^{p-1} \le c(p) \|w\|_{W_2^1(\Omega)}^{3(p-2)/2}  \| w\|_{L_2(\Omega)}^{(4-p)/2}
$$
to complete the proof.
\end{proof}

%%%%%%%%%%%%%%%%%%%%%%%
\begin{lem}\label{L2b}
Suppose \eqref{D}. Let $\varepsilon>0$, $\kappa\in (0,1)$, and $q>2$. For each $v\in \overline{S}_q(\kappa)$ and $F\in L_2(\Omega)$, the weak solution $\Phi\in W_{2,\mathcal{B}}^1(\Omega)$ to \eqref{231} belongs to the Hilbert space $X(\Omega)$ defined by
$$
X(\Omega) := \left\{ w \in W_{2,\mathcal{B}}^1(\Omega)\ :\ \partial_\eta w \in W_2^1(\Omega) \right\}\ ,
$$
and there is $c_4(\kappa,\ve)>0$ such that
\begin{equation}
\|\Phi\|_{W_2^1(\Omega)} + \|\partial_\eta \Phi\|_{W_2^1(\Omega)} \le c_4(\kappa,\ve) \left( \|\Phi\|_{L_2(\Omega)} + \|F\|_{L_2(\Omega)} \right)\ . \label{cocv1}
\end{equation}
\end{lem}
%%%%%%%%%%%%%%%%%%%%%%%

\begin{proof}[{\bf Proof}]
We first recall that, due to the continuous embedding of $W_q^2(D)$ in $W_\infty^1(D)$, there is a positive constant $c_0>0$ depending only on $q$ and $D$ such that, for $v\in \overline{S}_q(\kappa)$,
\begin{equation}\label{lazy0}
1+v(x,y)\ge \kappa\ ,\quad (x,y)\in D\ ,\quad\text{ and }\;\; \|v\|_{C^1(\bar D)}\le \frac{c_0}{\kappa}\ .
\end{equation}

Consider $F\in L_2(\Omega)$ and denote the corresponding weak solution to \eqref{231} by $\Phi \in W_{2,\mathcal{B}}^1(\Omega)$. We begin with an estimate for $\Phi$ in $W_2^1(\Omega)$ and first infer from \eqref{231} and the Divergence Theorem that
\begin{align*}
\int_\Omega F \Phi \,\rd (x,y,\eta) & = - \int_\Omega \Phi \mathcal{L}_v \Phi \,\rd (x,y,\eta) \\ 
& = \ve^2 \|P_x\|_{L_2(\Omega)}^2 + \ve^2 \|P_y\|_{L_2(\Omega)}^2 + \|P_\eta\|_{L_2(\Omega)}^2 \\
& \quad - \ve^2 \int_\Omega \left[ \frac{\partial_x v}{1+v} \Phi \partial_x\Phi + \frac{\partial_y v}{1+v} \Phi \partial_y\Phi - \eta \frac{|\nabla v|^2}{(1+v)^2} \Phi \partial_\eta\Phi \right] \,\rd(x,y,\eta)
\end{align*}
with
\begin{equation*}
P_x := \partial_x \Phi - \eta \frac{\partial_x v}{1+v} \partial_\eta\Phi\ , \quad P_y := \partial_y \Phi - \eta \frac{\partial_y v}{1+v} \partial_\eta\Phi\ , \quad P_\eta := \frac{\partial_\eta \Phi}{1+v}\ . %\label{defP}
\end{equation*}
Note that \eqref{lazy0} ensures that
\begin{align}
& \|\nabla \Phi\|_{L_2(\Omega)}^2 \le c(\kappa) \left[ \|P_x\|_{L_2(\Omega)}^2 + \|P_y\|_{L_2(\Omega)}^2 + \|P_\eta\|_{L_2(\Omega)}^2 \right]\ . \label{lazy0a}
%& \|P_x\|_{L_2(\Omega)}^2 + \|P_y\|_{L_2(\Omega)}^2 + \|P_\eta\|_{L_2(\Omega)}^2 \le c(\kappa) \|\nabla \Phi\|_{L_2(\Omega)}^2\ . \label{lazy0b}
\end{align}
Using \eqref{lazy0} and Cauchy-Schwarz' and Young's inequalities, we obtain
\begin{align*}
& \ve^2 \|P_x\|_{L_2(\Omega)}^2 + \ve^2 \|P_y\|_{L_2(\Omega)}^2 + \|P_\eta\|_{L_2(\Omega)}^2 \\
& \qquad = \int_\Omega F \Phi \,\rd (x,y,\eta)  + \ve^2 \int_\Omega \left[ \frac{\partial_x v}{1+v} \Phi P_x + \frac{\partial_y v}{1+v} \Phi P_y \right] \,\rd(x,y,\eta) \\
& \qquad \le \|F\|_{L_2(\Omega)} \|\Phi\|_{L_2(\Omega)} + \ve^2 \|\Phi\|_{L_2(\Omega)} \left\| \frac{|\nabla v|}{1+v} \right\|_{L_\infty(D)} \left[ \|P_x\|_{L_2(\Omega)} + \|P_y\|_{L_2(\Omega)} \right] \\
& \qquad \le \frac{\ve^2}{2} \left[ \|P_x\|_{L_2(\Omega)}^2 + \|P_y\|_{L_2(\Omega)}^2 \right] + c(\kappa,\ve) \left[ \| F \|_{L_2(\Omega)}^2 + \|\Phi\|_{L_2(\Omega)}^2 \right]\ ,
\end{align*}
whence
$$
\ve^2 \|P_x\|_{L_2(\Omega)}^2 + \ve^2 \|P_y\|_{L_2(\Omega)}^2 + \|P_\eta\|_{L_2(\Omega)}^2 \le c(\kappa,\ve) \left[ \| F \|_{L_2(\Omega)}^2 + \|\Phi\|_{L_2(\Omega)}^2 \right]\ . 
$$
Combining \eqref{lazy0a} with the above estimate gives
\begin{equation}
\| \Phi\|_{W_2^1(\Omega)}^2 \le c(\kappa,\ve) \left[ \| F \|_{L_2(\Omega)}^2 + \|\Phi\|_{L_2(\Omega)}^2 \right]\ . \label{lazy1}
\end{equation}

\medskip

We now turn to an estimate on $\partial_\eta \Phi$ in $W_2^1(\Omega)$ which is established first for smooth functions $v$, the constants appearing in the estimates depending only on $q$ and $\kappa$. Indeed, assume first that, besides being in $\overline{S}_q(\kappa)$, the function $v$ also belongs to $\overline{S}_\infty(\kappa')$ for some $\kappa'\in (0,1)$. Then $\Phi\in W_{2,\mathcal{B}}^2(\Omega)$ by Lemma~\ref{L2a} and, setting
$$
\zeta_x := \partial_{x}\partial_\eta\Phi\ , \quad \zeta_y := \partial_{y}\partial_\eta\Phi\ , \quad \zeta_\eta := \partial_\eta^2 \Phi\ ,
$$
it follows from Lemma~\ref{below} below that
\begin{equation}
\begin{split}
\int_\Omega \partial_x^2 \Phi\ \partial_\eta^2 \Phi \,\rd(x,y,\eta) & = \int_\Omega
|\zeta_x|^2 \,\rd(x,y,\eta)\ , \\
 \int_\Omega \partial_y^2 \Phi\ \partial_\eta^2 \Phi \,\rd(x,y,\eta) & = \int_\Omega |\zeta_y|^2 \,\rd(x,y,\eta)\ .
 \end{split} \label{lazy2}
\end{equation} 
We then infer from \eqref{231} and \eqref{lazy2} that
\begin{align*}
- \int_\Omega F\ \zeta_\eta \,\rd(x,y,\eta) & = \int_\Omega \partial_\eta^2\Phi\ \mathcal{L}_v\Phi \,\rd(x,y,\eta) \\
& = \ve^2 \|Q_x\|_{L_2(\Omega)}^2 + \ve^2 \|Q_y\|_{L_2(\Omega)}^2 + \|Q_\eta\|_{L_2(\Omega)}^2 \\
& \quad + \ve^2 \int_\Omega \eta \left( \frac{2|\nabla v|^2 - (1+v)\Delta v}{(1+v)^2} \right) \partial_\eta \Phi \ \partial_\eta^2\Phi \,\rd(x,y,\eta) \ ,
\end{align*}
where
\begin{equation*}
Q_x := \zeta_x - \eta \frac{\partial_x v}{1+v} \zeta_\eta\ , \quad Q_y := \zeta_y - \eta \frac{\partial_y v}{1+v} \zeta_\eta \ , \quad Q_\eta := \frac{\zeta_\eta}{1+v}\ . %\label{defQ}
\end{equation*}
Owing to \eqref{lazy0} we note that
\begin{align}
& \|\nabla\partial_\eta\Phi\|_{L_2(\Omega)}^2 \le c(\kappa) \left[ \|Q_x\|_{L_2(\Omega)}^2 + \|Q_y\|_{L_2(\Omega)}^2 + \|Q_\eta\|_{L_2(\Omega)}^2 \right]\ . \label{lazy2a} 
%& \|Q_x\|_{L_2(\Omega)}^2 + \|Q_y\|_{L_2(\Omega)}^2 + \|Q_\eta\|_{L_2(\Omega)}^2 \le c(\kappa) \|\nabla\partial_\eta\Phi\|_{L_2(\Omega)}^2\ . \label{lazy2b}
\end{align}
Since $\partial_\eta \Phi \ \partial_\eta^2\Phi = \partial_\eta \left( |\partial_\eta \Phi|^2 \right)/2$, using integration by parts to handle the last term on the right-hand side of the above identity leads us to
\begin{eqnarray}
& & \ve^2 \|Q_x\|_{L_2(\Omega)}^2 + \ve^2 \|Q_y\|_{L_2(\Omega)}^2 + \|Q_\eta\|_{L_2(\Omega)}^2 \nonumber\\
& = & - \int_\Omega F \zeta_\eta \,\rd(x,y,\eta) + \frac{\ve^2}{2} \int_\Omega \left( \frac{2|\nabla v|^2 - (1+v) \Delta v}{(1+v)^2} \right) |\partial_\eta \Phi|^2 \,\rd(x,y,\eta) \nonumber \\
& & \quad - \frac{\ve^2}{2} \int_D \left( \frac{2|\nabla v|^2 -  (1+v) \Delta v}{(1+v)^2} \right) |\partial_\eta \Phi(x,y,1)|^2 \,\rd(x,y)\ . \label{lazy3}
\end{eqnarray}
We now estimate successively the three terms on the right-hand side of \eqref{lazy3} and begin with the first one which is the easiest. Indeed, it follows from \eqref{lazy0} and Cauchy-Schwarz' and Young's inequalities that
\begin{equation}
\left| - \int_\Omega F \zeta_\eta \,\rd(x,y,\eta) \right| \le \|(1+v) F\|_{L_2(\Omega)} \|Q_\eta\|_{L_2(\Omega)} \le \frac{1}{4} \|Q_\eta\|_{L_2(\Omega)}^2 + c(\kappa) \|F \|_{L_2(\Omega)}^2\ . \label{lazy4}
\end{equation} 
Next, introducing $q':= q/(q-1)\in (1,2)$, we infer from H\"older's and Gagliardo-Nirenberg inequalities that 
\begin{align*}
& \left| \int_\Omega \left( \frac{2|\nabla v|^2 - (1+v) \Delta v}{(1+v)^2} \right) |\partial_\eta \Phi|^2 \,\rd(x,y,\eta) \right| \\
& \quad \le \frac{1}{\kappa^2} \left[ 2 \|\nabla v\|_{L_\infty(D)} \|\nabla v\|_{L_q(D)} + \| 1+v\|_{L_\infty(D)} \|\Delta v\|_{L_q(D)} \right] \| \partial_\eta \Phi\|_{L_{2q'}(\Omega)}^2 \\
& \quad \le c(\kappa) \|\partial_\eta\Phi\|_{W_2^1(\Omega)}^{3/q} \|\partial_\eta\Phi \|_{L_2(\Omega)}^{(2q-3)/q} \\
& \quad = c(\kappa) \left( \|\partial_\eta \Phi\|_{L_2(\Omega)}^2 + \|\nabla\partial_\eta \Phi\|_{L_2(\Omega)}^2 \right)^{3/2q} \|\partial_\eta\Phi \|_{L_2(\Omega)}^{(2q-3)/q} \,.%\|\zeta_x\|_{L_2(\Omega)}^2 + \|\zeta_y\|_{L_2(\Omega)}^2 + \|\zeta_\eta\|_{L_2(\Omega)}^2 \right)^{3/2q} \|\partial_\eta\Phi \|_{L_2(\Omega)}^{(2q-3)/q} \ .
\end{align*}
Using \eqref{lazy2a} and Young's inequality we end up with
\begin{align}
& \left| \int_\Omega \left( \frac{2|\nabla v|^2 - (1+v) \Delta v}{(1+v)^2} \right) |\partial_\eta \Phi|^2 \,\rd(x,y,\eta) \right| \nonumber \\ 
& \quad \le \frac{1}{2} \|Q_x\|_{L_2(\Omega)}^2 + \frac{1}{2} \|Q_y\|_{L_2(\Omega)}^2 + \frac{1}{2\ve^2} \|Q_\eta\|_{L_2(\Omega)}^2 + c(\kappa,\ve) \|\partial_\eta\Phi \|_{L_2(\Omega)}^2\ . \label{lazy5}
\end{align}
Similarly, since $2q'\in (2,4)$, H\"older's and Young's inequalities combined with \eqref{lazy2a} and Lemma~\ref{le.lazy} entail that
\begin{align}
& \left| \int_D \left( \frac{2|\nabla v|^2 -  (1+v) \Delta v}{(1+v)^2} \right) |\partial_\eta \Phi(x,y,1)|^2 \,\rd(x,y) \right| \nonumber \\
& \quad \le \frac{1}{\kappa^2} \left[ 2 \|\nabla v\|_{L_\infty(D)} \|\nabla v\|_{L_q(D)} + \|1+ v\|_{L_\infty(D)}\|\Delta v\|_{L_q(D)} \right] \| \partial_\eta \Phi(\cdot,\cdot,1)\|_{L_{2q'}(D)}^2 \nonumber \\
& \quad \le c(\kappa) \|\partial_\eta \Phi \|_{L_2(\Omega)}^{(q-2)/q} \|\partial_\eta \Phi \|_{W_2^1(\Omega)}^{(q+2)/q} \nonumber \\
& \quad \le \frac{1}{2} \|Q_x\|_{L_2(\Omega)}^2 + \frac{1}{2} \|Q_y\|_{L_2(\Omega)}^2 + \frac{1}{2\ve^2} \|Q_\eta\|_{L_2(\Omega)}^2 + c(\kappa,\ve) \|\partial_\eta\Phi \|_{L_2(\Omega)}^2\ . \label{lazy6}
\end{align}
Inserting \eqref{lazy4}-\eqref{lazy6} in \eqref{lazy3} leads us to 
$$
\ve^2 \|Q_x\|_{L_2(\Omega)}^2 + \ve^2 \|Q_y\|_{L_2(\Omega)}^2 + \|Q_\eta\|_{L_2(\Omega)}^2 \le c(\kappa,\ve) \left( \|F \|_{L_2(\Omega)}^2 + \|\partial_\eta\Phi \|_{L_2(\Omega)}^2 \right)\ ,
$$
from which we deduce, thanks to \eqref{lazy1} and \eqref{lazy2a}, that \eqref{cocv1} holds true. 

Since the estimate \eqref{cocv1} does not depend on the regularity of $v$, the fact that it extends to all functions in $\overline{S}_q(\kappa)$ follows by a classical approximation argument.
\end{proof}

It remains to prove the auxiliary result used in \eqref{lazy2} which is recalled in the following lemma.

%%%%%%%%%%%%%%%%%%%%%%%%%%%
\begin{lem}\label{below}
If $\Phi\in W_{2,\mathcal{B}}^2(\Omega)$, then 
\begin{equation*}
\begin{split}
\int_\Omega \partial_x^2 \Phi\ \partial_\eta^2 \Phi \,\rd(x,y,\eta) & = \int_\Omega
|\partial_{x}\partial_\eta\Phi|^2 \,\rd(x,y,\eta)\ , \\
 \int_\Omega \partial_y^2 \Phi\ \partial_\eta^2 \Phi \,\rd(x,y,\eta) & = \int_\Omega | \partial_{y}\partial_\eta\Phi|^2 \,\rd(x,y,\eta)\ .
 \end{split} 
\end{equation*}
\end{lem}
%%%%%%%%%%%%%%%%%%%%%%%%%%

\begin{proof}
Since $\Omega=D\times (0,1)$ and $D$ is a convex subset of $\R^2$, the projection $pr_2(D)$ of $D$ onto the $y$-axis as well as the sections
$$
D_{[y]}:=\{x\in\R\,:\, (x,y)\in D\}\ ,\quad y\in pr_2(D)\,,
$$ 
are intervals. Moreover, the Fubini-Tonelli Theorem (together with Nikodym's characterization of Sobolev spaces via absolutely continuous functions, see \cite[Theorem~2.1.4]{Zi89}) implies that for a.a. $y\in  pr_2(D)$, the function $\Phi(\cdot,y,\cdot)$ belongs to $W_2^2(D_{[y]}\times (0,1))$ with $\Phi(\cdot,y,\cdot)=0$ on $\partial(D_{[y]}\times (0,1))$ since $\Phi\in W_{2,\mathcal{B}}^2(\Omega) \subset C(\overline{\Omega})$. Thus, $D_{[y]}\times (0,1)$ being a rectangle, we may apply \cite[Lemma~4.3.1.2]{Grisvard} to conclude that
$$
\int_{D_{[y]}\times (0,1)} \partial_x^2\Phi(x,y,\eta)\partial_\eta^2\Phi(x,y,\eta)\,\rd (x,\eta)=\int_{D_{[y]}\times (0,1)} \vert\partial_x\partial_\eta\Phi(x,y,\eta)\vert^2\,\rd (x,\eta) 
$$
for a.a. $y\in  pr_2(D)$. Recalling that $pr_2(\Omega)= pr_2(D)$ is measurable and 
$$
\Omega_{[y]}:=\{(x,\eta)\in\R\times\R\,:\, (x,y,\eta)\in \Omega\}=D_{[y]}\times (0,1)\,,
$$ 
the first assertion follows by integrating the above identity on $pr_2(\Omega)$ with respect to $y$ and using the Fubini-Tonelli Theorem.
The second assertion is analogous.
\end{proof}

To recover the full $W_2^2$-regularity of the solution to \eqref{231} we need to have slightly smoother functions $v$.

%%%%%%%%%%%%%%%%%%%%%%%
\begin{prop}\label{PLazy}
Suppose \eqref{D}. Let $\varepsilon>0$, $\kappa\in (0,1)$, and $q\ge 3$. For each $v\in \overline{S}_q(\kappa)$ and $F\in L_2(\Omega)$, the weak solution $\Phi\in W_{2,\mathcal{B}}^1(\Omega)$ to \eqref{231} belongs to $W_{2,\mathcal{B}}^2(\Omega)$ and there  is $c_5(\kappa,\ve)>0$ such that
\begin{equation}
\|\Phi\|_{W_2^2(\Omega)} \le c_5(\kappa,\ve) \|F\|_{L_2(\Omega)} \ . \label{lazy7}
\end{equation}
\end{prop}
%%%%%%%%%%%%%%%%%%%%%%%

\begin{proof}[{\bf Proof}]
Consider $F\in L_2(\Omega)$ and denote the corresponding weak solution to \eqref{231} by $\Phi$. Introducing 
\begin{equation*}
\begin{split}
J(x,y,\eta) & := 2\e^2\ \eta\ \left( \frac{\partial_x v}{1+v} \right)(x,y)\ \partial_x\partial_\eta \Phi(x,y,\eta) + 2\e^2\ \eta\ \left(  \frac{\partial_y v}{1+v} \right)(x,y)\ \partial_y\partial_\eta \Phi(x,y,\eta)\\
& \quad + \left( 1 - \frac{1+\e^2\eta^2 |\nabla v|^2}{(1+v)^2} \right)(x,y)\ \partial_\eta^2 \Phi(x,y,\eta) \\
& \quad - \e^2\ \eta\ \left( 2\ \frac{|\nabla v|^2}{(1+v)^2} - \frac{\Delta v}{1+v} \right)(x,y)\ \partial_\eta \Phi(x,y,\eta)\ ,
\end{split}
\end{equation*}
it follows from \eqref{231} that $\Phi$ solves
$$
\e^2\ \partial_x^2 \Phi + \e^2\ \partial_y^2 \Phi + \partial_\eta^2 \Phi = J \;\;\text{ in }\;\; \Omega\ , \qquad \Phi=0 \;\;\text{ on }\;\; \partial\Omega\ .  
$$
Moreover, Lemma~\ref{L2b} and the continuous embeddings of $W_q^2(D)$ in $W_\infty^1(D)$ and $W_2^1(\Omega)$ in $L_6(\Omega)$ guarantee that $J$ belongs to $L_2(\Omega)$ with
$$
\| J\|_{L_2(\Omega)}^2 \le c(\kappa,\ve) \left( \|\Phi\|_{L_2(\Omega)}^2  + \|F\|_{L_2(\Omega)}^2  \right)\ .
$$
We then infer from \cite[Theorem~3.2.1.2]{Grisvard} that $\Phi\in W_2^2(\Omega)$ and inspecting the proof of \cite[Theorem~3.2.1.2]{Grisvard} along with \cite[Theorem~3.1.3.1 \& Lemma~3.2.1.1]{Grisvard} ensures that 
$$
\| \Phi\|_{W_2^2(\Omega)} \le c(\kappa,\ve) \| J\|_{L_2(\Omega)}\ .
$$
We have thus shown that $\Phi\in W_2^2(\Omega)$ and satisfies
$$
\| \Phi\|_{W_2^2(\Omega)} \le c(\kappa,\ve) \left( \|\Phi\|_{L_2(\Omega)}^2  + \|F\|_{L_2(\Omega)}^2  \right)\ .
$$ 
Finally, since the embeddings of $W_2^2(\Omega)$ in $W_2^1(\Omega)$ and $W_q^2(D)$ in $C^1(\bar D)$ are compact, we may proceed as in the proof of \cite[Eq.~(19)]{ELW1} to derive \eqref{lazy7}.
\end{proof}

\medskip

After these preliminary steps we are in a position to prove Proposition~\ref{L1}.

\begin{proof}[{\bf Proof of Proposition~\ref{L1}}]
For $v\in S_q(\kappa)$ and $(x,y,\eta)\in\Omega$, we set 
$$
f_v(x,y,\eta) := \mathcal{L}_v \eta = \varepsilon^2\ \eta \left[ 2\ \frac{|\nabla v(x,y)|^2}{1+v(x,y)^2} - \frac{\Delta v(x,y)}{1+v(x,y)} \right]\ .
$$
Since $W_q^2(D)$ is embedded in $W_\infty^1(D)$, the function $f_v$ belongs to $L_2(\Omega)$ with 
\begin{equation}
\| f_v \|_{L_2(\Omega)} \le c_6(\kappa,\varepsilon)\ , \label{y19}
\end{equation}
and Proposition~\ref{PLazy} ensures that there is a unique solution $\Phi_v\in W_{2,\mathcal{B}}^2(\Omega)$ to
\begin{equation*}
-\mathcal{L}_v\Phi_v = f_v\ \text{ in }\ \Omega\ ,\ \qquad \Phi_v =0\   \text{ on }\ \partial\Omega\ , %\label{23a} 
\end{equation*}
satisfying
\begin{equation}
\|\Phi_v\|_{W_2^2(\Omega)}\le c_5(\kappa,\varepsilon)\ {  \|f_v\|_{L_2(\Omega)}}\ . \label{gaston}
\end{equation}
Setting $\phi_v(x,y,\eta)=\Phi_v(x,y, \eta)+\eta$ for $(x,y,\eta) \in \overline{\Omega}$, the function $\phi_v$ obviously solves \eqref{23}-\eqref{24} while we deduce from \eqref{y19} and \eqref{gaston} that
\begin{equation}\label{c4}
\|\phi_v\|_{W_2^2(\Omega)}\le c_7(\kappa,\varepsilon)\ .
\end{equation}
We next define a bounded linear operator $\mathcal{A}(v)\in\mathcal{L}\big(W_{2,\mathcal{B}}^2(\Omega), L_2(\Omega)\big)$ by
$$
\mathcal{A}(v)\Phi:=-\mathcal{L}_v\Phi\ ,\quad \Phi\in W_{2,\mathcal{B}}^2(\Omega)\ .
$$
Proposition~\ref{PLazy} guarantees that $\mathcal{A}(v)$ is invertible with inverse $\mathcal{A}(v)^{-1} \in \mathcal{L}\big(L_2(\Omega),W_{2,\mathcal{B}}^2(\Omega))$ satisfying
\begin{equation}
\left\| \mathcal{A}(v)^{-1} \right\|_{\mathcal{L}\big(L_2(\Omega),W_{2,\mathcal{B}}^2(\Omega))} \le c_5(\kappa,\varepsilon)\ . \label{y25}
\end{equation}
We then note that
\begin{equation}
\|\mathcal{A}(v_1)-\mathcal{A}(v_2)\|_{\mathcal{L}(W_{2,\mathcal{B}}^2(\Omega),L_2(\Omega))}\le c_8(\kappa,\varepsilon)\ \|v_1-v_2\|_{W_q^2(D)}\ ,\quad v_1, v_2\in S_q(\kappa)\ ,\label{y26}
\end{equation}
which follows from the definition of $\mathcal{L}_v$ and the continuity of pointwise multiplication
$$
W_q^1(D)\cdot W_q^1(D)\hookrightarrow W_q^1(D) \hookrightarrow L_\infty(D)
$$
except for the terms involving $\partial_x^2 v_i$ and $\partial_y^2 v_i$, $i=1,2$, where continuity of pointwise multiplication
$$
L_q(D)\cdot W_2^1(\Omega)\hookrightarrow L_2(\Omega)\ ,
$$
the latter being true thanks to the continuous embedding of $W_2^1(\Omega)$ in $L_6(\Omega)$ and the choice $q\ge 3$. Similar arguments also show that
\begin{equation}
\left\| f_{v_1} - f_{v_2} \right\|_{L_2(\Omega)} \le c_9(\kappa,\ve) \ \|v_1-v_2\|_{W_q^2(D)}\ ,\quad v_1, v_2\in S_q(\kappa)\ . \label{y26a}
\end{equation}
Now, for $v_1, v_2\in S_q(\kappa)$, we infer from \eqref{y25} and \eqref{y26} that
\begin{equation*}
\|\mathcal{A}(v_1)^{-1}   -\mathcal{A}(v_2)^{-1}\|_{\mathcal{L}(L_2(\Omega),W_{2,\mathcal{B}}^2(\Omega))} \le  c_{10}(\kappa,\varepsilon)\ \|v_1-v_2\|_{W_q^2({D})}\ ,
\end{equation*}
which, combined with \eqref{y19}, \eqref{y25}, \eqref{y26a}, and the observation that $\phi_v=\mathcal{A}(v)^{-1}  f_v$ implies \eqref{RR}. 

\medskip

Since $W_2^{1/2}(D)$ embeds continuously in~$L_4(D)$, pointwise multiplication
\begin{equation}\label{m*}
W_2^{1/2}(D)\cdot W_2^{1/2}(D)\hookrightarrow L_2(D)
\end{equation}
is continuous and hence, invoking \cite[Chapter~2, Theorem~5.5]{Necas67} (since $\Omega=D\times (0,1)$ is a bounded Lipschitz domain),  the mapping
\begin{equation}
S_q(\kappa)\rightarrow L_2(D)\ ,\quad v\mapsto \big\vert\partial_\eta \phi_v(\cdot,\cdot,1)\big\vert^2 \label{c12a}
\end{equation}
is bounded and globally Lipschitz continuous. Thanks to the continuity of the embedding of $W_q^1(D)$ in $L_\infty(D)$, the mapping
\begin{equation}
S_q(\kappa)\rightarrow W_q^1(D)\ ,\quad v\mapsto \frac{1+\varepsilon^2 |\nabla v|^2}{(1+v)^2} \label{c12b}
\end{equation}
is bounded and globally Lipschitz continuous with a Lipschitz constant depending only on $\kappa$ and $\varepsilon$, and the Lipschitz continuity of $g_\ve$ stated in Proposition~\ref{L1} follows at once from those of the mappings in \eqref{c12a} and~\eqref{c12b}. Finally, to prove that $g_\ve$ is analytic, we note that $S_q(\kappa)$ is open in $W_{q,\mathcal{B}}^2(D)$ and that the mappings 
$$
\mathcal{A}: S_q(\kappa)\rightarrow\mathcal{L}(W_{2,\mathcal{B}}^2(\Omega),L_2(\Omega))\quad \text{and}\quad [v\mapsto f_v]: S_q(\kappa)\rightarrow L_2(\Omega)
$$ 
are analytic. The analyticity of the inversion map $\ell\mapsto\ell^{-1}$ for bounded operators implies that also the mapping 
$$
S_q(\kappa)\rightarrow W_2^2(\Omega) \ ,\quad v\mapsto \phi_v=\mathcal{A}(v)^{-1}  f_v
$$ 
is analytic, and the assertion follows as above from \eqref{c12a} and \eqref{c12b}.
\end{proof}

Let us point out that $g_\ve$ also maps $S_q(\kappa)$ into a (non-Hilbert) space of more regularity.

%%%%%%%%%%%%%%%%%%%%%%%
\begin{cor}\label{C11}
Suppose \eqref{D}. Let $\varepsilon>0$, $\kappa\in (0,1)$, and $q\ge 3$. For $p\in [1,2)$ and $\sigma\in (0,1/2)$ with $\sigma < (2-p)/p$, the mapping
$$
g_\ve: S_q(\kappa)\longrightarrow W_{p,\mathcal{B}}^{\sigma}(D)\ ,\quad v\longmapsto \frac{1+\e^2 |\nabla v|^2}{(1+v)^2}\  \vert\partial_\eta\phi_v(\cdot,\cdot,1)\vert^2
$$
is analytic, globally Lipschitz continuous, and bounded. 
\end{cor}
%%%%%%%%%%%%%%%%%%%%%%%

\begin{proof} Given $p\in [1,2)$ and $\mu\in (0,1/2]$ with $\mu < (2-p)/p$, we may replace \eqref{m*} in the proof of Proposition~\ref{L1} by the pointwise multiplication
\begin{equation}\label{m**}
W_2^{1/2}(D)\cdot W_2^{1/2}(D)\hookrightarrow W_p^{\mu}(D)\,,
\end{equation}
which is continuous according to \cite[Theorem~4.1 \& Remarks~4.2(d)]{AmannMultiplication}, and deduce again from \cite[Chapter~2, Theorem~5.5]{Necas67} that the mapping
\begin{equation}
S_q(\kappa)\rightarrow W_p^{\mu}(D)\ ,\quad v\mapsto \big\vert\partial_\eta \phi_v(\cdot,\cdot,1)\big\vert^2 \label{c12aa}
\end{equation}
is globally Lipschitz continuous. Moreover, due to the continuity of the pointwise multiplication
$$
W_q^{1}(D)\cdot W_p^{\mu}(D)\hookrightarrow W_p^{\sigma}(D)
$$
for any $\sigma<\mu$, see again \cite[Theorem~4.1]{AmannMultiplication}, the claimed Lipschitz continuity of $g_\ve$  follows at once from \eqref{c12b} and \eqref{c12aa}, the proof of the analyticity of $g_\ve$ being the same as in Proposition~\ref{L1}.
\end{proof}

\bigskip

%%%%%%%%%%%%%%%%%%%%%%%
%%%%%%%%%%%%%%%%%%%%%%%
\section{Proof of Theorem~\ref{A}}\label{Sec3}
%%%%%%%%%%%%%%%%%%%%%%%
%%%%%%%%%%%%%%%%%%%%%%%

We now turn to the proof of Theorem~\ref{A}. Let us first note that, using the notation from the previous section and noticing that $\partial_x\phi_u(x,y,1)=\partial_y\phi_u(x,y,1)=0$ for $(x,y)\in D$ due to $\phi_u(x,y,1)=1$ by~\eqref{24}, the boundary value problem \eqref{u}-\eqref{bcpsi} can be stated equivalently as a single nonlinear equation for $u$ of the form
\begin{align}
\partial_t u+\beta\Delta^2 u-\big(\tau +a\|\nabla u\|_2^2\big)\Delta u  &= {-\lambda}\, \frac{1+\e^2 |\nabla u|^2}{(1+u)^2}\, \vert\partial_\eta\phi_u(\cdot,\cdot,1)\vert^2\  \label{33}
\end{align}
subject to the boundary conditions~\eqref{bcu} and the initial condition~\eqref{ic}.
To analyze this equation it is useful to write it as an abstract Cauchy problem to which semigroup theory applies.
Let $\xi>0$ be fixed such that $7/3 < 4\xi <4$ and consider $u^0\in W_{2,\mathcal{B}}^{4\xi}(D)$ with $u^0>-1$ on $D$. Owing to the continuous embedding of $W_2^{4\xi}(D)$ in $W_3^2(D)$ and $C(\bar{D})$, there are $\bar{c}>1$ and $\kappa\in (0,1/2)$ such that 
$$
\| w \|_{W_3^2(D)} \le \bar{c}\ \| w \|_{W_2^{4\xi}(D)}\ , \qquad w\in W_2^{4\xi}(D)\ ,
$$
and $u^0 \in S_3(2\kappa)$ with $\|u^0\|_{W_2^{4\xi}(D)} \le 1/(2\kappa)$. Next, define the operator $A\in \mathcal{L}(W_{2,\mathcal{B}}^4(D),L_2(D))$ by
\begin{equation}\label{AD}
Aw:=\beta\Delta^2 w-\tau\Delta w\ ,\quad w\in W_{2,\mathcal{B}}^4(D)\ ,
\end{equation}
and recall that $-A$ generates an exponentially decaying analytic semigroup on $L_2(D)$ with
\begin{equation}\label{sgexp}
\|e^{-tA}\|_{\mathcal{L}(W_{2,D}^{4\xi}(D))} + \|e^{-tA}\|_{\mathcal{L}(W_{2,D}^{1}(D))}+ t^\xi\,\|e^{-tA}\|_{\mathcal{L}(L_2(D),W_{2,D}^{4\xi}(D))}\, \le  M e^{-\omega t}\ ,\quad t\ge 0\ ,
\end{equation}
for some $\omega>0$ and $M\ge 1$. Set $\kappa_0 := \kappa/(M\bar{c})\in (0,\kappa)$ and introduce the mapping $h$ defined by
$$
h(v):=-\lambda g_\ve (v)+ a \|\nabla v\|_2^2\, \Delta v\ ,\quad v\in \overline{S}_3(\kappa_0)\ .
$$
According to Proposition~\ref{L1}, $h$ is well-defined on $\overline{S}_3(\kappa_0)$ and there is a constant $C_1(\kappa):=c_1(\kappa_0,\ve)$ with
\begin{equation}
\|h(v)\|_{L_2(D)} \le C_1(\kappa) \left( \lambda + a \|\nabla v \|_2^2 \right) \|v\|_{W_3^2(D)}\ , \qquad v\in \overline{S}_3(\kappa_0)\ , \label{lazyh} 
\end{equation}
and 
\begin{equation}\label{hh}
\| h(v_1)-h(v_2)\|_{L_2(D)}\le C_1(\kappa) \left[ \lambda + a \left( \|\nabla v_1\|_2 + \|\nabla v_2\|_2 \right) \right] \| v_1-v_2\|_{W_3^2(D)}\,,\quad v_1, v_2 \in \overline{S}_3(\kappa_0) \, .
\end{equation} 
Consequently,  we may rewrite \eqref{33} as a semilinear Cauchy problem for $u$ of the form
\begin{equation}\label{CP}
\partial_t u + A u= h(u)\ ,\quad t>0\ ,\qquad u(0)=u^0\ .
\end{equation}
Choosing $\vartheta > 2M \|\nabla u^0\|_2$ we define for $T>0$ the complete metric space 
$$
\mathcal{V}_T(\kappa,\vartheta) := \left\{ v\in C([0,T],\overline{S}_3(\kappa_0))\ :\ \sup_{t\in [0,T]} \|\nabla v(t)\|_{L_2(D)} \le \vartheta \right\} 
$$
endowed with the metric induced by the norm in $C([0,T],\overline{S}_3(\kappa_0))$. Arguing as in the proofs of \cite[Theorem~1]{ELW1} and \cite[Proposition~3.2~(iii)]{Bending1} with the help of \cite[Chapter~II, Theorem~5.3.1]{AmBk}, we readily deduce from \eqref{sgexp}, \eqref{lazyh}, and \eqref{hh} that the mapping $\Lambda$, given by
$$
\Lambda(v)(t):=e^{-tA} u^0 + \int_0^t e^{-(t-s)A} h\big(v(s)\big)\,\rd s\, ,\quad t\in [0,T]\, ,\quad v\in C([0,T],\overline{S}_3(\kappa_0))\,,
$$
defines a contraction on $\mathcal{V}_T(\kappa,\vartheta)$ provided that $T:=T(\lambda,\kappa,\vartheta)>0$ is sufficiently small. Thus $\Lambda$ has a unique fixed point $u$ in $\mathcal{V}_T(\kappa,\vartheta)$ which is a solution to \eqref{CP} with the regularity properties stated in \eqref{reg}.  This readily implies parts~(i)-(ii) of Theorem~\ref{A}. To prove the global existence claimed in part~(iii) of Theorem~\ref{A}  we use the fact that $\omega>0$ in \eqref{sgexp} and proceed as in the proof of  \cite[Theorem~1]{ELW1} to establish that there is $m_*(\kappa)>0$ such that $\Lambda$ is a contraction on $\mathcal{V}_T(\kappa,\vartheta)$ for {\it each} $T>0$ provided that  $\lambda +a\vartheta\le m_*(\kappa)$. Thus $u$ is a global solution to \eqref{CP}, whence Theorem~\ref{A}~(iii).

Finally, to prove part~(iv) of Theorem~\ref{A}, let $D$ be a disc in $\R^2$. Introducing then $\tilde u$ as an arbitrary rotation of $u$ with respect to $(x,y)\in D$, the rotational invariance of \eqref{psi} with respect to $(x,y)\in D$ implies that $\tilde u$ is again a solution to \eqref{CP} and thus coincides with $u$ by uniqueness. This yields Theorem~\ref{A}~(iv).

%%%%%%%%%%%%%%%%%%%%%%%
\begin{rem}\label{R1}
Besides the proof of Proposition~\ref{L1}, which requires a different approach, the range of the map $g_\varepsilon$ identified in Proposition~\ref{L1} and Corollary~\ref{C11} appears to be the main difference between the two-dimensional case considered in \cite{ELW1} and the  three-dimensional case considered herein. Since the space $W_p^{2+\sigma}(D)$ {\em does not} embed in $W_3^{2}(D)$ under the constraints $\sigma\in (0,1/2)$ and $\sigma<(2-p)/p$ with $p\in (1,2)$ stated in Corollary~\ref{C11}, we cannot identify $\Lambda$ from the previous proof as a contraction on a suitable space when dealing with the second-order problem~$\beta=0$.
\end{rem}
%%%%%%%%%%%%%%%%%%%%%%%

%%%%%%%%%%%%%%%%%%%%%%%
%%%%%%%%%%%%%%%%%%%%%%%
\section{Proof of Theorem~\ref{Stat}}\label{Sec4}
%%%%%%%%%%%%%%%%%%%%%%%
%%%%%%%%%%%%%%%%%%%%%%%

Let $\kappa\in (0,1)$ and note that $W_2^4(D)$ embeds continuously in $W_3^2(D)$. Define the operator $A$ as in \eqref{AD} and note that, since $A\in \mathcal{L}(W_{2,\mathcal{B}}^4(D),L_2(D))$ is invertible, the mapping
$$
\mathcal{F}:\R\times \big(W_{2,\mathcal{B}}^4(D)\cap S_3(\kappa)\big)\rightarrow W_{2,\mathcal{B}}^4(D)\ ,\quad (\lambda,v)\mapsto v+\lambda A^{-1}g_\ve(v) - a\|\nabla v\|_2^2A^{-1}\Delta v
$$
is analytic with $\mathcal{F}(0,0)=0$ and $D_v\mathcal{F}(0,0) =\mathrm{id}_{W_{2,\mathcal{B}}^4(D)}$. Now, the Implicit Function Theorem ensures the existence of $\delta=\delta(\kappa,\ve)>0$ and a branch $(U_\lambda)_{\lambda\in [0,\delta)}$ in $W_{2,\mathcal{B}}^4(D)$ such that $\mathcal{F}(\lambda,U_\lambda)=0$ for all $\lambda\in [0,\delta)$.  Denoting the solution to \eqref{psi}-\eqref{bcpsi} corresponding to $U_\lambda$ by $\Psi_\lambda\in W_2^2(\Omega(U_\lambda))$, the pair $(U_\lambda,\Psi_\lambda)$ is thus for each $\lambda\in (0,\delta)$ a stationary solution to \eqref{u}-\eqref{bcpsi}. This proves the existence part of Theorem~\ref{Stat}. 

We next use the Principle of Linearized Stability \cite[Theorem~9.1.1]{Lunardi} as in \cite[Theorem~3]{ELW1} to obtain the following proposition, which completes the proof of Theorem~\ref{Stat}:

%%%%%%%%%%%%%%%%%%%%%%%
\begin{prop}
Let $\lambda\in (0,\delta(\kappa,\ve))$. There are $\omega_0,r_0,R>0$ such that for each initial value $u^0\in W_{2,\mathcal{B}}^{4}(D)$ with $\|u^0-U_\lambda\|_{W_{2,\mathcal{B}}^{4}(D)} <r_0$, the solution $(u,\psi)$ to \eqref{u}-\eqref{bcpsi} exists globally in time and
\begin{equation*}\label{est}
\|u(t)-U_\lambda\|_{W_{2,\mathcal{B}}^{4}(D)}+\|\partial_t u(t)\|_{L_{2}(D)} \le R e^{-\omega_0 t} \|u^0-U_\lambda\|_{W_{2,\mathcal{B}}^{4}(D)}\ ,\quad t\ge 0\ .
\end{equation*}
\end{prop}
%%%%%%%%%%%%%%%%%%%%%%%

Using Corollary~\ref{C11} we can actually improve the regularity of $U_\lambda$ slightly. Indeed, Corollary~\ref{C11} and the fact that $U_\lambda\in W_2^4(D)$ entail that the right-hand side of
$$
A U_\lambda = - \lambda g_\ve(U_\lambda) + a\ \|\nabla U_\lambda\|_2^2\ \Delta U_\lambda
$$
belongs to $W_{p,\mathcal{B}}^\sigma(D)$ for $\sigma\in (0,1/2)$, $\sigma<(2-p)/p$, and $p\in (1,2)$. Now the invertibility of the operator $A$  in $\mathcal{L}(W_{p,\mathcal{B}}^{4+\sigma}(D),W_{p,\mathcal{B}}^\sigma(D))$ implies  that $U_\lambda$ actually belongs to $W_{p,\mathcal{B}}^{4+\sigma}(D)$ for $\lambda\in [0,\delta(\kappa,\ve))$.

%%%%%%%%%%%%%%%%%%%%%%%
%%%%%%%%%%%%%%%%%%%%%%%
\section{Proof of Theorem~\ref{StatDisc}}\label{Sec5}
%%%%%%%%%%%%%%%%%%%%%%%
%%%%%%%%%%%%%%%%%%%%%%%
 
In this section, we restrict ourselves to the particular case where $D$ is a disc in $\R^2$ and assume for simplicity that $D$ is the unit disc in $\R^2$. 

\medskip

To prove Theorem~\ref{StatDisc}~(i) we argue as in the proof of Theorem~\ref{A}~(iv). Let $\lambda\in (0,\delta(\kappa,\ve))$. Since \eqref{psi} is rotationally invariant with respect to $(x,y)\in D$, any rotation of $U_\lambda$ is again a solution to \eqref{u} in $W_{2,\mathcal{B}}^4(D)\cap S_3(\kappa)$ and thus coincides with $U_\lambda$ by the uniqueness assertion of Theorem~\ref{Stat}. The non-positivity of $U_\lambda$ then follows from~\cite{LW_Boggio} since $g_\ve(U_\lambda)<0$.

\medskip

Next, the proof of the non-existence statement in Theorem~\ref{StatDisc}~(ii) is a straightforward adaptation of that of \cite[Theorem~1.7~(ii)]{Bending1} which is based on a nonlinear version of the eigenfunction technique. We thus omit the proof herein but mention that it heavily relies on the existence of a positive eigenfunction $\zeta_1$ of the operator $\beta\Delta^2 - \tau \Delta$ in $W_{2,\mathcal{B}}^4(D)$ associated to a positive eigenvalue $\mu_1$. This result follows from \cite[Theorem~4.7]{LW_Boggio} and we emphasize that the positivity of $\zeta_1$ is due to a variant of Boggio's principle \cite{Bo05} and requires $D$ to be a disc. Moreover, the assumptions that $D$ is a disc and that the sought-for steady state $u$ is radially symmetric also guarantee that $u$ is negative in $D$ by \cite[Theorem~1.4]{LW_Boggio}. Finally, the radial symmetry plays again an important r\^ole in deriving suitable estimates for the auxiliary function $\mathcal{U}$ defined by 
$$
- \Delta \mathcal{U} = u \;\;\text{ in }\;\; D\ , \qquad \mathcal{U} = 0 \;\;\text{ on }\;\; \partial D\ .
$$
Indeed, $\mathcal{U}$ is obviously radially symmetric and its profile $\bar{\mathcal{U}}$ defined by $\bar{\mathcal{U}}(\sqrt{|x|^2+|y|^2}) := \mathcal{U}(x,y)$ for $(x,y)\in D$ satisfies
$$
|\partial_r \bar{\mathcal{U}}(r)| \le \frac{r}{2}\ , \qquad |\partial_r^2 \bar{\mathcal{U}}(r)| \le \frac{3}{2}\ , \quad r\in [0,1]\ .
$$
The previous bounds follow from the fact that $-1<u<0$ in $D$ and explicit integration of the ordinary differential equation solved by $\bar{\mathcal{U}}$.

%%%%%%%%%%%%%%%%%%%%%%%
%%%%%%%%%%%%%%%%%%%%%%%

%%%%%%%%%%%%%%%%%%%%%%%
%%%%%%%%%%%%%%%%%%%%%%%

\end{document}